\documentclass[12pt, letterpaper]{article}
\usepackage{amsfonts, amsmath,amsthm}
\usepackage[utf8]{inputenc}
\usepackage{comment}
\numberwithin{equation}{section}

\newtheorem{corollary}[equation]{Corollary}
\newtheorem{proposition}[equation]{Proposition}
\newtheorem{definition}[equation]{Definition}
\newtheorem{conjecture}[equation]{Conjecture}
\newtheorem{remark}[equation]{Remark}

\numberwithin{equation}{section}

\def\calO{\mathcal{O}}

\def\bbP{\mathbb{P}}

\def\bbX{\mathbb{X}}

\begin{document}

\title{Cremona Orbits in $\bbP^4$ and Applications}
\author{
Olivia Dumitrescu \thanks{The first author is supported by NSF grant DMS1802082.} \\
\small University of North Carolina at Chapel Hill \\ \small Chapel Hill, NC 27599-3250 \\
\small Simion Stoilow Institute of Mathematics\\ \small Romanian Academy\\ \small 010702 Bucharest, Romania \\
\and
Rick Miranda \\
\small Colorado State University \\ \small Fort Collins, CO 80523 USA}
\maketitle

\begin{abstract}
This article is motivated by the authors interest in the geometry of the Mori dream space
$\bbP^4$ blown up in$8$ general points. 
In this article we develop the necessary technique
for determining Weyl orbits of linear cycles for the four-dimensional case,
by explicit computations in the Chow ring of the resolution of the standard Cremona transformation.
In particular, we close this paper with applications to the question of
the dimension of the space global sections of effective divisors having at most $8$ base points.
\end{abstract}

\section{Introduction}

Let $X^n_s$ be projective space $\bbP^n$ blown up at $s$ general points.
Motivated by the study of the dimensionality problem for effective divisors on $X^n_{s}$,
we analyze the standard Cremona action on $X^{4}_{8}$ and give several applications.
We establish first the terminology we use throughout the paper.
We call a \emph{Weyl line/Cremona line (Weyl hyperplane respectively)} to be
the orbit under the Weyl group action
of a line passing through two of the $s$ points
(hyperplane passing through $n$ of the points).
In dimension two the \emph{Weyl lines} are also known in the literature as $(-1)$ curves;
via a theorem of Nagata \cite[Theorem 2a]{nagata1}
they can be described via numerical properties as irreducible classes with selfintersection $-1$ and anticanonical degree $1$.
In \cite{dpr} the authors noticed that Nagata's work can be generalized and similar numerical properties via Dolgachev-Mukai bilnear form are equivalent to
{\it Weyl divisors}. In dimension three the Weyl group action on curves was analyzed by Laface - Ugaglia in \cite{LU}.
Finally, in arbitrary dimension,  the Weyl group action on curves in  $X^n_{s}$
and their connection to  $(-1)$-curves introduced by Clemens, Friedman and Kontsevich
is analyzed by the two authors in a forthcomming paper \cite{DM}. 

In the planar case the Gimigliano-Harbourne-Hirschowitz conjecture, still open,
predicts that the dimension of the space of global sections of an effective divisor
depends on the Euler characteristic and the multiplicity of containment of Weyl lines in the base locus of the divisor.
In $\mathbb{P}^3$, the conjecture of Laface-Ugaglia \cite{LU}
predicts that this dimension depends on  the multiplicity of containtment of Weyl lines, Weyl hyperplanes,
and Weyl orbit of the unique quadric in $X^3_s$ passing through nine general points.

In general, for a small number of points, $X^n_{n+2}$,
it was proved this dimension depends on the Euler characteristic
and multiplicity of contaiment of \emph{linear cycles}
spanned by the fixed points in the base locus of the divisor $D$ as in \cite[Theorem 2.3]{bdp1}.
Moreover, the birational geometry of the space $X^n_{n+3}$,
studied in several publications (e.g. \cite{AC}, \cite{AM}, \cite{bdp3}),
namely the effective and movable cone of divisors,
their Mori chamber decompositions
together with the dimension of space of global sections
is determined by secant varieties to the the rational normal curve of degree $n$
passing through $n+3$ general points
together with their joins.
In general, the case $X^n_{n+4}$ seems to be mysterious.

We dedicate this paper to study $X^4_{8}$, which is a Mori Dream Space,
whose birational geometry is not totally explained in the literature. 
In this paper together with \cite{bdp4} we define and classify the varieties
that determine combinatorial data describing the geometry of $X^4_{8}$.


The two spaces $X_{2,8}$ and $X_{4,8}$, are related by Gale duality as in \cite{Mu05}.
The precise relation between $X_{2,8}$ and $X_{4,8}$ was established in the following theorem of Mukai
(semistability refers to semistability in the sense of Gieseker-Maruyama):
{\it $X_{4,8}$ is isomorphic to the moduli space of rank $2$ torsion free sheaves $F$ on $X_{2,8}$
for which $c_1(F)=-K_S$ and $c_2(F)=2$.}
Via Mukai's correspondence, Casagrande et al. describe in \cite{Cas2}
the five types surfaces in $X_{4,8}$ playing a special role in the Mori program.
In this paper we rediscover these surfaces as \emph{Weyl planes}, defined below analogously to Weyl lines and hyperplanes.

The Weyl group of $X^{4}_{8}$ is generated by the standard Cremona transformations
together with permutations of the base points. 
In order to define and construct Weyl planes,
we introduce $Y^{4}_{8}$ to denote the blow up of $X^{4}_{8}$ along all lines joining any two points
and the eight rational normal curves of degree $4$ passing through $7$ points.
(These curves are all disjoint in $X^4_8$.)
\begin{definition}\label{wp}
A \emph{Weyl plane} is the Weyl orbit of the proper transform of a plane through three fixed points
under the blow up of the three lines joining any two points  in  $Y^{4}_{8}$ .
\end{definition}
It is important to remark that {\it Weyl planes} live on the space $Y^{4}_{8}$.
We emphasize that this orbit is different (in the Chow ring) than the Weyl orbit of {\it planes through three points}.
Moreover, in  \cite{bdp4} the authors introduce and classify the notions of {\it Weyl curves} and {\it Weyl surfaces} in $X^4_{8}$ as the 
intersection of two distinct {\it Weyl divisors}
that are orthogonal with respect to the Dolgachev-Mukai bilinear pairing. Since the classification of {\it Weyl surfaces} \cite{bdp4} in $X^4_{8}$ is the same with the classification of Proposition \ref{P4orbits} we can deduce that the two definitions of {\it Weyl planes} \eqref{wp} and {\it Weyl surfaces} \cite{bdp4} are equivalent in $X^4_{8}$.
By definition {\it Weyl lines} coincide with {\it Weyl curves} in the projective plane $X^{2}_{s}$, but the explicit relation between the two definitions in general will be studied in a differet paper.

In this paper Corollaries \ref{P4multiplicityonly} and \ref{surfacecremonaformulas}
enable us to determine the Weyl action on\\
(a)  $1$-cycles (i.e. curves) on the Chow ring of blow up of $X^{4}_{s}$ \\
(b) $2$-cycles (i.e. surfaces) on the Chow ring of $Y^{4}_{8}$.

As a consequence, Proposition \ref{P4orbits} determines the complete list of {\it Weyl planes} and {\it Weyl divisors} on $X^{4}_{8}$,
and it also  gives the formulas for all {\it Weyl lines} on $X^{4}_{s}$,
(for arbitrary number of fixed points $s$).
In particular, for $X^{4}_{8}$ \emph{the only Weyl lines are lines through two fixed points
and the rational normal curve of degree $4$ passing through $7$ of the $8$ points}.
In fact, in forthcoming paper \cite{DM} we prove this statement holds for all Mori Dream Spaces.
Let $Q_i$ denote {\it Weyl line} of degree $4$ (the rational normal quartic) skipping only the $i$th point.
In particular, we prove that on $X^{4}_{8}$ there are $5$ types of Weyl planes
(modulo permutation of points),
matching computations in \cite[Theorem 8.7]{Cas2} and \cite{bdp4}:

\begin{itemize}
\item The $56$ planes $S_1(ijk)$ through three of the $8$ points ($p_i$, $p_j$, $p_k$);
it has multiplicity one along the three lines $L_{ij}$, $L_{ik}$, $L_{jk}$
\item The $56$ cubic surfaces $S_3(i,j)$ triple at $p_i$, passing through all other points except $p_j$;
it has multiplicity one along the lines $L_{ik}$ for $k \neq i$ and $k \neq, j$, and along $Q_j$
\item The $56$ sextic surfaces $S_6(ijk)$ passing through $p_i$, $p_j$, and $p_k$ and triple at the other five points;
it has multiplicity one along all lines joining two of the five points, and along $Q_i$, $Q_j$, and $Q_k$
\item The $28$ surfaces $S_{10}(ij)$ of degree $10$ having two points $p_i$ and $p_j$ of multiplicity $6$ and triple at the other six points;
it has multiplicity $3$ along the line $L_{ij}$, multiplicity one along all lines $L_{ik}$ and $L_{jk}$ for $k\neq i,j$,
and multiplicity one along the curves $Q_k$ for $k \neq i,j$
\item The $8$ surfaces $S_{15}(i)$ of degree $15$ having one point $p_i$ with multiplicity $3$ and having multiplicity $6$ at the other seven points;
it has multiplicity one along all lines $L_{jk}$ for $j,k \neq i$, multiplicity one along each $Q_j$ for $j\neq i$, and multiplicity $3$ along $Q_i$.

\end{itemize}

In addition to the multiplicities at the points $p_i$,
the reader will note that for all of these surfaces we also compute the multiplicities along the lines $L_{ij}$
and along the rational normal quartics (through $7$ of the $8$ points).
This is important for computations in the Chow ring:
unless one takes into account that these surfaces have multiplicity along these curves,
one does not fully capture the intersection behavior of these surfaces after one blows up the points
(and in general the curves and surfaces that appear as base loci of linear systems of divisors).
It is also critical for computations of the dimensions of the linear systems:
it is one of the principles of this article that the multiplicities along these curves
must be taken into account in determining the difference between the virtual dimension
and the actual dimension of linear systems.
Indeed, for certain purposes it is useful to consider not only the blowup $X^4_8$
of $\bbP^4$ at the $8$ general points, but also then the further blowup $Y^4_8$
of all of the proper transforms of the lines $L_{ij}$ and the rational normal quartics $Q_k$;
these are easily seen to be disjoint in $X^4_8$ and therefore $Y^4_8$ is smooth.

\begin{remark}
In paper \cite{bdp4} the authors use a different notation for the Chow ring basis. For example $\{h, e_i, e_{ij}\}$ and $\{h^1, e_i^1\}$ of \cite{bdp4} represent here $\{S, S_i, P_{ij}\}$ and $\{l, l_i\}$ respectively.  In \cite{bdp4} surfaces denoted above by $S_1(ijk)$,  $S_3(i,j)$,  $S_6(ijk)$, $S_{10}(ij)$, $S_{15}(i)$ are denoted by $H_{ijk}$, $S^{3}_{i, \widehat{j}}$, $S^{6}_{ijk}$, $S^{10}_{ij}$, $S^{15}_{i}$ respectively.
\end{remark}

We predict that the birational geometry of $X^{4}_{8}$ is determined not only by Weyl hyperplanes,
but also Weyl lines and Weyl planes classified in Proposition \ref{P4orbits}.
Finally, in Section \ref{applicationssection} we present applications to the vanishing conjecture and dimensionality problem.

\subsection*{Acknowledgements} The collaboration  was partially supported by NSF grant DMS1802082.

\section{The standard Cremona transformation and its resolution}
The standard Cremona transformation of $\bbP^n$ can be elegantly factored into a series of blowups
at the proper transforms of the coordinate linear spaces, followed by a series of symmetric blowdowns.

Fix coordinates $[x_0:x_1:\cdots:x_n]$ in $\bbP^n$, and consider the standard Cremona involution
\[
[x_0:x_1:\cdots:x_n] \longrightarrow [x_0^{-1}:x_1^{-1}:\cdots:x_n^{-1}]
\]
which simply inverts all the coordinates.
This is well-defined on the torus where all coordinates are non-zero,
and has fundamental locus the union of the coordinate hyperplanes.
The transformation is relatively straightforward to resolve in a sequence of blowups and blowdowns, as follows.

Let $p_0,p_1,\ldots,p_n$ be the coordinate points of $\mathbb{P}^n$.
For an index set $I \subset \{0,1,\ldots,n\}$, denote by $L_I$ the linear span of the coordinate points indexed by $I$:
$L_I = \text{span} \{p_i\;|\; i \in I\}$.
We have that $\dim L_I = |I|-1$.

We set $\bbX^n_0 = \mathbb{P}^n$,
and define $\pi_j:\bbX^n_j \to \bbX^n_{j-1}$ to be the blowup of the proper transforms of all $L_I$ with $|I|=j$.
Hence $\pi_1$ is the blowup of all the coordinate points in $\mathbb{P}^n$;
$\pi_2$ is the blowup of the (proper transforms of the) coordinate lines $L_{ij}$; etc.
Note that the sequence of blowups stops with $\pi_{n-1}$, the blowup of the codimension two coordinate linear spaces,
creating the space $\bbX^n_{n-1}$.
We will denote by $E_I$ the exceptional divisor created when $L_I$ is blown up.
$E_I$ is created on $\bbX^n_{|I|}$,
and we will use the notation $E_I$ for the proper transform on subsequent blowups too.
If $|I|=n$, then $L_I$ is a coordinate hyperplance in $\bbP^n$;
we will denote its proper transform in $\bbX^n_{n-1}$ by $E_I$ as well.

We note at this point that, on $\bbX^n_{n-1}$, the nature and configuration of the divisors $E_I$
are completely symmetric, with respect to taking complements;
in other words, we have an isomorphism of $\bbX^n_{n-1}$ that switches the roles of $E_I$ and $E_J$
when $I$ and $J$ are complementary in $\{0,1,\ldots,n\}$.
Hence we can reverse the sequence of blowups with the complementary divisors,
and blow down to $\bbP^n$ ``the other way'':
first blow down the $E_I$ with $|I|=2$,
then the $E_I$ with $|I|=3$, etc.,
finishing by blowing down the proper transforms of the coordinate hyperplanes $E_{|I|}$
with $|I|=n$.
This is the resolution of the birational involution.

We note that:
\begin{itemize}
\item On $\bbX^n_{|I|-1}$ when the $L_I$ are blown up, they are all disjoint.
\item Each linear space $L_I$ experiences a sequence of blowups (by the earlier blowups);
on $\bbX^n_{|I|-1}$, the proper transform of each $L_I$ is isomorphic to $\bbX^{|I|-1}_{|I|-2}$.
\item By induction, this proper transform has both the hyperplane divisor class $H$
(the pullback of the hyperplane divisor class on $\bbX^{|I|-1}_0 = \bbP^{|I|-1}$)
and its Cremona involution image $H'$.
\item On $\bbX^{|I|-1}_{|I|-2}$, the normal bundle of the proper transform of $L_I$
is isomorphic to 
\[\calO(-H')^{\oplus n-|I|+1}.\]
\item Since the normal bundle of the proper transform of $L_I$ splits as a direct product of identical line bundles,
when $E_I$ is created on $\bbX^n_{|I|}$, it is isomorphic to a product $\bbX^{|I|-1}_{|I|-2} \times \bbP^{n-|I|}$.
\item $E_I$ experiences further blowups on its way to $\bbX^n_{n-1}$, and there it is isomorphic to
$\bbX^{|I|-1}_{|I|-2} \times \bbX^{n-|I|}_{n-|I|-1}$,
where it has normal bundle isomorphic to the tensor product
of the anti-Cremona-hyperplane bundles coming from the two factors.
\end{itemize}

This construction generalizes the familiar construction of the quadratic Cremona transformation of $\bbP^2$,
which is obtained by blowing up the three coordinate points $L_0$, $L_1$, and $L_2$ (obtaining $\bbX^2_1$)
and then blowing down the three coordinate lines $L_{01}$, $L_{02}$, and $L_{12}$.

\section{The case of three-space}
For three-space, the sequence of iterated blowups in this case involve two sets of blowups:
\[
\bbX^3_2 \stackrel{\pi_2}{\longrightarrow} \bbX^3_1 \stackrel{\pi_1}{\longrightarrow} \bbX^3_0 = \mathbb{P}^3
\]
where $\pi_1$ blows up the four coordinate points $p_i = L_i$ and $\pi_2$ blows up the six proper transforms of the coordinate lines $L_{ij}$.
The exceptional divisors $E_i$ start out as $\bbP^2$'s in $\bbX^3_1$,
and then are further blown up to become isomorphic to $\bbX^2_1$'s in $\bbX^3_2$.
The coordinate lines start in $\bbP^2$ having normal bundle of bidegree $(1,1)$;
after blowing up the two coordinate points on each, the proper transforms have normal bundles with bidegree $(-1,-1)$ in $\bbX^3_1$.
They are then blown up to $E_{ij} \cong \bbP^1\times\bbP^1$ in $\bbX^3_2$.
Finally the coordinate hyperplanes $L_{ijk}$ are each blown up three times by $\pi_1$,
and then not blown up further by $\pi_2$, and so arrive at $\bbX^3_2$ as surfaces isomorphic to $\bbX^2_1$.

The blowing down proceeds by blowing down the $E_{ij}$ via the other ruling,
which blows down each $L_{ijk}$ to a $\bbP^2$;
one then blows down each of these to points, finishing the process.

If one is interested in intersection phenomena related to these coordinate subspaces,
the Chow ring is the appropriate tool; it is useful primarily for recording two different kinds of phenomena.
One is \emph{containment} (with multiplicity) by a given subvariety of one of the blow-up centers.
In $\bbP^3$, for divisors, this is the multiplicity of the divisor at one of the coordinate points,
and the multiplicity of containment along one of the coordinate lines.
For curves, this is the multiplicity of the curve at one of the coordinate points.
For a divisor written in the form
$D = dH - \sum_i m_i E_i - \sum_{ij} n_{ij} E_{ij}$,
the coefficient $d$ is the degree; $m_i$ is the multiplicity at the coordinate point $L_i$;
and $m_{ij}$ is the multiplicity along the line $L_{ij}$.

The other phenomena which the Chow ring coefficients can record
is the higher-dimensional \emph{contact} that the given subvariety may have with one of the blow-up centers.
(Higher-dimensional contact in the sense of higher than expected dimension.)
In $\bbP^3$, for surfaces, this is not relevant for the coordinate points and lines;
higher-dimensional contact is containment with multiplicity.
This is also true for curves with respect to the points: the only phenomenon is that of containment.
However with curves, one can have additional contact with the lines, without containment.

The Chow ring of $\bbX^3_2$ is not difficult to compute;
all the relevant tools are presented in \cite{EH}, chapters 9 and 13.
The codimension zero classes are one-dimensional, generated by $[\bbX^3_2]$ itself;
the codimension three classes are also one-dimensional, generated by the class $[p]$ of a point.
The codimension one classes are freely generated by the pullback $H$ of the hyperplane class,
and the exceptional divisors $E_i$ and $E_{ij}$.

In codimension two, the group $A^2(\bbX^3_2)$ contain the following elements.
The pullback of the general line class in $\bbP^3$ will be denoted by $\ell$.
The general line class inside the exceptional divisor $E_i$ will be denoted by $\ell_i$.
The exceptional divisor $E_{ij}$ is isomorphic to $\bbP^1\times\bbP^1$,
and contributes a priori two curve classes:
the class $f_{ij}$ of the fiber of the blowup $\pi_2$,
and the class $g_{ij}$ which is the horizontal ruling of $E_{ij}$.
These are not independent though in $A^2(\bbX^3_2)$; it is an exercise to check that
\[
g_{ij} = f_{ij} + \ell - \ell_i - \ell_j
\]
and that this is the only relation in $A^2$.

For a curve class $C$ written as
$C = d\ell - \sum_im_i\ell_i - \sum_{ij} n_{ij}f_{ij}$,
the coefficient $d$ is the degree,
$m_i$ is the multiplicity of $C$ at the coordinate point $L_i$,
and $n_{ij}$ is the additional contact of $C$ with the coordinate line $L_{ij}$
(over and above the contact implied by the multiplicities at the two coordinate points on $L_{ij}$).

We have the following,
where we use typical $\delta$-notation:
$\delta_{I,J} = 1$ if $I \subseteq J$ and $0$ otherwise.

\begin{proposition}\label{ChowringX32}
\item[(a)] A basis for the Chow ring of $X^3_2$ is given by:
\[
\begin{matrix}
A^0: & [X^3_2] \\
A^1: & H, E_0, E_1, E_2, E_3, E_{01}, E_{02}, E_{03}, E_{12}, E_{13}, E_{23} \\
A^2: & \ell, \ell_0, \ell_1, \ell_2, \ell_3, f_{01},f_{02}, f_{03}, f_{12}, f_{13}, f_{23} \\
A^3: & p
\end{matrix}
\]
\item[(b)] Multiplication of these basis elements are given by:
\[
\begin{matrix}
A^1\cdot A^1 & H & E_i & E_{ij} \\
H & \ell & 0 & f_{ij} \\
E_k & 0 & -\ell_i\delta_{ik} & f_{ij}\delta_{k,ij} \\
E_{kl} & f_{kl} & f_{kl}\delta_{i,kl} & (-2f_{ij} -\ell+ \ell_i+\ell_j)\delta_{ij,kl}
\end{matrix}
\]
\[
\begin{matrix}
A^1\cdot A^2 & H & E_i & E_{ij} \\
\ell     & p & 0 & 0 \\
\ell_k  & 0 & -p \delta_{i,k} &  0 \\
f_{kl}  & 0 &  0 & -p \delta_{ij,kl} \\
\end{matrix}
\]
\end{proposition}

The Cremona involution extends to an involution $\phi$ on the Chow ring; we denote the image of the involution using a superscript prime:
\begin{itemize}
\item $[\bbX^3_2] \leftrightarrow [\bbX^3_2]$
\item $H \leftrightarrow H' = 3H-2\sum_i E_i - \sum_{ij} E_{ij}$
\item $E_l \leftrightarrow E_l' = L_{ijk} = H - E_i-E_j-E_k - E_{ij}-E_{ik}-E_{jk}$ for $i,j,k \neq l$
\item $E_{ij} \leftrightarrow E_{ij}' = E_{kl}$ for $k,l \neq i,j$.
\item $\ell \leftrightarrow \ell' = 3\ell-\sum_i\ell_i$
\item $\ell_i \leftrightarrow \ell_i' = 2\ell - \sum_{j \neq i} \ell_j$
\item $f_{ij} \leftrightarrow f_{ij}' = g_{kl} = f_{kl}+\ell - \ell_k - \ell_l$ for $k,l \neq i,j$.
\item $p \leftrightarrow p$.
\end{itemize}

We leave it to the reader to check that this is a ring automorphism, and is an involution.

\begin{proposition}\label{X32Cremonaformulas}
\mbox{\;}
\begin{itemize}
\item[(a)] Let $D = dH - \sum_i m_i E_i - \sum_{ij} n_{ij} E_{ij}$ be a general class in $A^1(\bbX^3_2)$.
Then the Cremona image $D'$ of $D$ under the involution is
$D' = d'H - \sum_i m_i' E_i - \sum_{ij} n_{ij}' E_{ij}$
where
\begin{align*}
d' &= D' \cdot \ell = D \cdot \ell' = D \cdot 3\ell-\sum_i\ell_i = 3d -\sum_i m_i; \\
m_i' &= D' \cdot \ell_i = D \cdot \ell_i' = D \cdot 2\ell - \sum_{j \neq i} \ell_j = 2d - \sum_{j \neq i} m_j; \\
n_{ij}' &= D' \cdot f_{ij} = D \cdot f_{ij}' = D \cdot f_{kl}+\ell - \ell_k - \ell_l = d + n_{kl} -m_k-m_l \\
&\;\;\mbox{ for }\;\; k,l \neq i,j
\end{align*}

\item[(b)] Let $C = d\ell - \sum_im_i\ell_i - \sum_{ij} n_{ij}f_{ij}$ be a general class in $A^2(\bbX^3_2)$.
Then the Cremona image $C'$ of $C$ under the involution is
$C' = d'\ell - \sum_i m_i' \ell_i - \sum_{ij} n_{ij}' f_{ij}$
where
\begin{align*}
d' &= C' \cdot H = C \cdot H' = C \cdot (3H-2\sum_i E_i - \sum_{ij} E_{ij})
= 3d - 2\sum_i m_i - \sum_{ij} n_{ij}; \\
m_i' &= C' \cdot E_i = C \cdot E_i' = C \cdot (H - \sum_{j\neq i}E_j -\sum_{j,k\neq i}E_{jk})
= d - \sum_{j \neq i} m_j - \sum_{j,k \neq i} n_{jk}; \\
n_{ij}' &= C' \cdot E_{ij} = C \cdot E_{ij}' = C \cdot E_{kl} = n_{kl}
\;\;\mbox{ for }\;\; k,l \neq i,j \\
\end{align*}

\end{itemize}

\end{proposition}

(In the computations above we abuse notation and give the multiplications as integers instead of integer multiples of the point class $p$.)

If one is in the position of not needing to consider the contact phenomena for curves, one can simplify the formulas as follows.

\begin{corollary}
The subspace of $A^2(\bbX^3_2)$ spanned by $\ell$ and the $\ell_i$ is invariant under the Cremona involution.
If $C = d\ell - \sum_im_i\ell_i$ is a general class in $A^2(\bbX^3_2)$ in this subspace,
then the Cremona image $C'$ of $C$ under the involution is
$C' = d'\ell - \sum_i m_i' \ell_i$
where
\begin{align*}
d' &= C' \cdot H = C \cdot H' = C \cdot (3H-2\sum_i E_i - \sum_{ij} E_{ij})
= 3d - 2\sum_i m_i; \\
m_i' &= C' \cdot E_i = C \cdot E_i' = C \cdot (H - \sum_{j\neq i}E_j -\sum_{j,k\neq i}E_{jk})
= d - \sum_{j \neq i} m_j; \\
\end{align*}
\end{corollary}

\section{The Chow Ring for the case of $\bbP^4$}
The sequence of iterated blowups in this case involve three sets of blowups:
\[
\bbX^4_3 \stackrel{\pi_3}{\longrightarrow} \bbX^4_2 \stackrel{\pi_2}{\longrightarrow} \bbX^4_1 \stackrel{\pi_1}{\longrightarrow} \bbX^4_0 = \mathbb{P}^4
\]
where $\pi_1$ blows up the five coordinate points $p_i = L_i$ to divisors $E_i$,
$\pi_2$ blows up the ten proper transforms of the coordinate lines $L_{ij}$ to $E_{ij}$,
and $\pi_3$ blows up the ten proper transforms of the coordinate planes $L_{ijk}$ to $E_{ijk}$.


We denote by $H$ the general hyperplane class in $\bbP^4$ (and all its pullbacks);
let us denote by $S = H^2$ the class of the general $2$-plane, and $\ell = H^3$ the class of the general line; the point class will be $p$ as usual.

In this section we'll present the Chow ring $A^*(\bbX^4_3)$,
proceeding through the sequence of three blowups.
In the starting fourfold $\bbX^4_0 \cong \bbP^4$,
the relevant subvarieties are simply the linear space $L_I$ for $I \subset \{0,1,2,3,4\}$.

After blowng up the points via $\pi_1$, we have:
\begin{itemize}
\item The divisors $E_i \cong \bbP^3$.
\item The proper transforms of the lines $L_{ij} \cong \bbP^1$.
\item The proper transforms of the $2$-planes $L_{ijk} \cong \bbX^2_1$.
\item The proper transforms of the hyperplanes $L_{ijk\ell} \cong \bbX^3_1$.
\end{itemize}

We now blow up with $\pi_2$ the proper transforms of the ten lines $L_{ij}$,
to the exceptional divisors $E_{ij}$, to obtain $\bbX^4_2$;
there, we have the following descriptions of the relevant subvarieties:

\begin{itemize}
\item The divisors $E_i \cong \bbX^3_1$.
\item The exceptional divisors $E_{ij} \cong \bbP^1\times\bbP^2$.
\item The $2$-planes $L_{ijk} \cong \bbX^2_1$.
\item The hyperplane threefolds $L_{ijk\ell} \cong \bbX^3_2$.
\end{itemize}

Finally we blow up the proper transforms of the ten surfaces $L_{ijk}$,
to the exceptional divisors $E_{ijk}$, to obtain $\bbX^4_3$;
there, the relevant subvarieties are:

\begin{itemize}
\item The divisors $E_i \cong \bbX^3_2$.
\item The divisors $E_{ij} \cong \bbP^1\times \bbX^2_1$.
\item The exceptional divisors $E_{ijk} \cong \bbX^2_1 \times \bbP^1$.
\item The hyperplane threefolds $L_{ijk\ell} \cong \bbX^3_2$.
\end{itemize}

The codimension one classes in $A^1(\bbX^4_3)$ are freely generated by
the pullback $H$ of the hyperplane class in $\bbP^4$
and the exceptional divisors $E_i$, $E_{ij}$, and $E_{ijk}$;
there are no relations among these.

In the group $A^2(\bbX^4_3)$ of codimension two classes,
we have the class $S = H^2$ of the pullback of a general $2$-plane in $\bbP^4$.
The other classes that will generate $A^2$ are supported in the exceptional divisors.

In $E_i$, which starts in $\bbX^4_1$ as a $\bbP^3$,
we have the general $2$-plane; pulled back to $\bbX^4_3$ this gives a class $S_i$ for each $i$.

The divisor $E_{ij}$ starts in $\bbX^4_2$ as isomorphic to the product $\bbP^1\times\bbP^2$.
This contributes two surface classes:
the fiber $\{\text{point}\}\times\bbP^2$ of the blowup,
and the product $\bbP^1\times\{\text{general line in }\bbP^2\}$.
Denote by $F_{ij}$ the pullback to $\bbX^4_3$ of the former, the fiber class;
and by $G_{ij}$ the pullback to $\bbX^4_3$ of the latter.

Finally the divisor $E_{ijk}$ is isomorphic to $\bbX^2_1\times\bbP^1$,
and contributes five surface classes.
One is $M_{ijk} = \bbX^2_1\times\{\text{point}\}$,
a cross-section of the blowup map.
The others come from products of curve classes in $L_{ijk}\cong \bbX^2_1$ with the fiber $\bbP^1$.
The curve classes in $L_{ijk}$ are generated by the pullback (from $\bbP^2$)
of the general line class $\ell_{ijk}$
and the three exceptional curves $e_{ijk,i}$, $e_{ijk,j}$, and $e_{ijk,k}$
which are (in $\bbX^4_2$) the intersection of $L_{ijk}$
with the three divisors $E_i$, $E_j$, and $E_k$ respectively.
These four classes give classes
$H_{ijk} = \ell_{ijk}\times\bbP^1$ and
$V_{ijk,i}$, $V_{ijk,j}$, and $V_{ijk,k}$
where $V_{ijk,i}$ comes from the product of $e_{ijk,i}\times\bbP^1$
and the same for the other two.

It is useful to introduce two new classes, for notational convenience.
These are:
\begin{equation}\label{PGFLambda}
P_{ij} = G_{ij}-F_{ij} \;\;\;\text{ and }\;\;\; \Lambda_{ijk} = 2H_{ijk}-V_{ijk,i}-V_{ijk,j}-V_{ijk,k};
\end{equation}
we note that $\Lambda_{ijk}$ is the pullback of the Cremona image of the line class on the $2$-plane $L_{ijk}$.
This will allow us to replace $G_{ij}$ by $P_{ij}$ among the generators for $A^2$.

There is a single relation among these codimension two classes beyond the definitional ones of (\ref{PGFLambda}).
It is that
\begin{equation}
M_{ijk}=S-S_i-S_j-S_k-P_{ij}-P_{ik}-P_{jk}+\Lambda_{ijk}.
\end{equation}

Finally we have the classes of the curves, the codimension three classes in $A^3(\bbX^4_3)$.
We again have the pullback $\ell$ of the general line class in $\bbP^4$,
and the classes $\ell_i$ of the general lines in the $E_i$.

The curve classes supported on $E_{ij}$
(which when it is created on $\bbX^4_2$ is isomorphic to $\bbP^1\times\bbP^2$)
are generated by the class $\ell_{ij} = \{\text{point}\}\times\{\text{general line in}\bbP^2\}$
and $h_{ij} = \bbP^1\times\{\text{point}\}$.

The curve classes coming from $E_{ijk}$
are the `horizontal' ones living in $L_{ijk}$, crossed with a point;
these we can denote again by $\ell_{ijk}$ and $e_{ijk,i}$, $e_{ijk,j}$, and $e_{ijk,k}$ as before.
The final one is a general fiber of the blowup $f_{ijk}$.

There are relations among these curve classes also; these are:
\begin{align}
h_{ij} &= \ell_{ij} + \ell - \ell_i - \ell_j; \;\;\; \ell_{ijk} = 2f_{ijk} + \ell -\ell_{ij}-\ell_{ik}-\ell_{jk}; \\
e_{ijk,i} &= f_{ijk} + \ell_i - \ell_{ij}-\ell_{ik}; \;\;\;
e_{ijk,j} = f_{ijk} + \ell_j - \ell_{ij}-\ell_{jk}; \;\;\;
e_{ijk,k} = f_{ijk} + \ell_k - \ell_{ik}-\ell_{jk}.
\end{align}
(Hence we can dispense with these to generate $A^3(\bbX^4_3)$.)

It is the case that, for a surface class $T$, one measures multiplicity along the line $L_{ij}$ by the intersection with $F_{ij}$, and one measures higher-dimensional contact with $L_{ij}$ by the intersection with $G_{ij}$.
Hence if the coefficients of $T$ include the terms $-mP_{ij} - nF_{ij}$,
then $m$ is the multiplicity of $T$ along the line
and $n$ is the additional contact of $T$ with the line,
so that one can read off these geometric phenomena from the coefficients directly.
($P$ and $F$ are the dual basis to $F$ and $G$ in $A^2$.)

We can similarly observe that a general surface class $T$
should meet the $2$-plane $L_{ijk}$ in a finite number of points.
The coefficients of $H_{ijk}$ and $V_{ijk,i}$, $V_{ijk,j}$, and $V_{ijk,k}$
(which generate the Picard group of the blown-up $L_{ijk}$)
record the higher-dimensional contact of a surface with $L_{ijk}$,
namely contact in a curve class rather than in a finite number of points.
Hence if the coefficients of $T$ include the terms
$-\alpha H_{ijk} + \beta_{ijk,i}V_{ijk,i} + \beta_{ijk,j}V_{ijk,j} + \beta_{ijk,k}V_{ijk,k}$
then the higher dimensional contact of $T$ with $L_{ijk}$
(away from the coordinate lines)
is a curve in the class 
$\alpha \ell_{ijk} - \beta_{ijk,i}e_{ijk,i} - \beta_{ijk,j}e_{ijk,j} - \beta_{ijk,k}e_{ijk,k}$.

Having described the generators for the Chow ring $A^*(\bbX^4_3)$,
we can now present the ring structure.
The computations are relatively straightforward,
using for example the formulas for the Chow rings of blowups
presented in \cite{EH}, chapter 13.
(The computation is iterative, first computing $A^*(\bbX^4_1)$,
then using that to compute $A^*(\bbX^4_2)$,
and finally $A^*(\bbX^4_3)$,)

\begin{proposition}\label{X43}
The Chow ring of $\bbX^4_3$ can be described as follows.
\begin{itemize}
\item[(a)] A basis for the Chow ring $A(\bbX^4_3)$ is given by the classes:
\[
\begin{matrix}
A^0: & [\bbX^4_3] = 1  \\
A^1: & H, E_i, E_{ij}, E_{ijk} \\
A^2: & S, S_i, P_{ij}, F_{ij}, H_{ijk}, V_{ijk,i}\\
A^3: & \ell, \ell_i, \ell_{ij},  f_{ijk} \\
A^4: & p \\
\end{matrix}
\]

\item[(b)] Multiplication of basis elements is given in the following tables.
\[
\begin{matrix}
A^1\cdot A^1 & H                      & E_i                          & E_{ij}                                   & E_{ijk} \\
H                   &    S                   & 0                             & F_{ij}                                   &  H_{ijk} \\
E_m               & 0                      & -S_i\delta_{i,m}        & F_{ij}\delta_{m,ij}               &  V_{ijk,m}\delta_{m,ijk}\\
E_{mn}          & F_{mn}            & F_{mn}\delta_{i,mn}  & -(P_{ij}+2F_{ij})\delta_{ij,mn}  & (H_{ijk}-V_{ijk,m}-V_{ijk,n})\delta_{mn,ijk} \\
E_{mnr}         & H_{mnr}           & V_{mnr,i}\delta_{i,mnr} & (H_{mnr}-V_{mnr,i}-V_{mnr,j})\delta_{ij,mnr} & -(M_{ijk}+\Lambda_{ijk})\delta_{ijk,mnr} \\
\end{matrix}
\]

\[
\begin{matrix}
A^1 \cdot A^2 & H            & E_i                             & E_{ij}                           & E_{ijk} \\
S                    & \ell          & 0                                & 0                                 &   f_{ijk}  \\
S_m               & 0             &  -\ell_i\delta_{i,m}       &  0                                 & f_{ijk}\delta_{m,ijk}  \\
P_{mn}        & \ell_{mn}   & \ell_{mn}\delta_{i,mn} & (-\ell_{ij} - \ell + \ell_i + \ell_j)\delta_{ij,mn} & -f_{ijk}\delta_{mn,ijk} \\
F_{mn}          & 0             & 0                                & -\ell_{ij}\delta_{ij,mn}    & f_{ijk}\delta_{mn,ijk}  \\
G_{mn}         & \ell_{mn}  & \ell_{mn}\delta_{i,mn} &  (-2\ell_{ij} - \ell + \ell_i + \ell_j)\delta_{ij,mn} & 0 \\
H_{mnr}         & f_{ijk}     &  0                               & f_{mnr}\delta_{ij,mnr}   & (-4f_{ijk}-\ell+\ell_{ij}+\ell_{ik}+\ell_{jk})\delta_{ijk,mnr}  \\
V_{mnr,m}      & 0            &  -f_{mnr}\delta_{i,m}   & f_{mnr}\delta_{m,ij}      & (-2f_{mnr}-\ell_m + \ell_{mn}+\ell_{mr})\delta_{ijk,mnr}\\
\end{matrix}
\]

\[
\begin{matrix}
A^1\cdot A^3 & H     & E_i                 & E_{ij}                  & E_{ijk} \\
\ell                 & p     & 0                    & 0                        &    0   \\
\ell_m             & 0     & -p\delta_{i,m} & 0                        &  0 \\
\ell_{mn}        & 0     & 0                    & -p \delta_{ij,mn}  & 0  \\
f_{mnr}         &   0    & 0                    &  0                       &     -p\delta_{ijk,mnr}    \\
\end{matrix}
\]

\[
\begin{matrix}
A^2\cdot A^2 & S & S_i & P_{ij} & F_{ij} & G_{ij} & H_{ijk} & V_{ijk,i} \\
S                   & p & 0  &  0  & 0 &  0 & 0  & 0 \\ 
S_m              & 0  &  -p\delta_{i,m} & 0  & 0 &  0  & 0  & 0 \\
P_{mn}         & 0 & 0 & p\delta_{ij,mn} &  -p\delta_{ij,mn} & 0 & 0 & 0 \\
F_{mn}         & 0  & 0 &  -p\delta_{ij,mn} & 0 &  -p\delta_{ij,mn} & 0 \\
G_{mn}        & 0 & 0 & 0 & -p\delta_{ij,mn} & -p\delta_{ij,mn} & 0 & 0 \\
H_{mnr}       & 0 & 0 & 0 & 0 & 0 & -p\delta_{ijk,mnr}  & 0 \\
V_{mnr,m}   &  0 & 0 & 0 & 0 & 0 &  0 & p\delta_{ijk,mnr}\delta_{i,m}\\
\end{matrix}
\]
\end{itemize}
\end{proposition}

\section{The Cremona Involution on $\bbP^4$}
Consider now the Cremona involution
\begin{align*}
[x_0:x_1:x_2:x_3:x_4] &\longrightarrow [\frac{1}{x_0}:\frac{1}{x_1}:\frac{1}{x_2}:\frac{1}{x_3}:\frac{1}{x_4}] \\
&= [x_1x_2x_3x_4:x_0x_2x_3x_4:x_0x_1x_3x_4:x_0x_1x_2x_4:x_0x_1x_2x_3]
\end{align*}
which lifts to a biregular automorphism of $\bbX^4_3$.
The induced action $\phi$ on the Chow ring $A(\bbX^4_3)$ is given as follows.

\begin{proposition}\label{CremonaP4}
\begin{align*}
\phi(H) &= 4H -3\sum_i E_i - 2\sum_{ij} E_{ij} - \sum_{ijk} E_{ijk} \\
\phi(E_i) & = [L_{jkmn\neq i}] = H - \sum_{m\neq i} E_m - \sum_{mn \neq i} E_{mn} - \sum_{mnr\neq i} E_{mnr} \\
\phi(E_{ij}) &= E_{mnr \neq i,j} \\
\phi(E_{ijk}) &= E_{mn \neq i,j,k} \\
\phi(S) &= 6S - 3\sum_i S_i - \sum_{ij} P_{ij} \\
\phi(S_m) &= 3S - 2\sum_{i\neq m} S_i - \sum_{ij\neq m} P_{ij} \\
\phi_(F_{mn}) &= M_{ijk\neq mn}
= S -S_i-S_j-S_k + F_{ij} + F_{ik} + F_{jk} -G_{ij} -G_{ik} -G_{jk} + \Lambda_{ijk} \\
&= S -S_i-S_j-S_k - P_{ij} - P_{ik} - P_{jk} + 2 H_{ijk} - V_{ijk,i} - V_{ijk,j} - V_{ijk,k} \\
\phi(G_{mn}) &= \Lambda_{ijk\neq mn} = 2H_{ijk} - V_{ijk,i}-V_{ijk,j}-V_{ijk,k} \\
\phi(P_{mn}) &= -S + S_i + S_j + S_k + P_{ij}+P_{ik}+P_{jk} (ijk \neq mn)\\
\phi(H_{mnr}) &= 2G_{ij} - (H_{ijm} - V_{ijm,i}-V_{ijm,j}) - (H_{ijn} - V_{ijn,i}-V_{ijn,j})- (H_{ijr} - V_{ijr,i}-V_{ijr,j}) \\
&= 2P_{ij}+ 2F_{ij} - (H_{ijm} - V_{ijm,i}-V_{ijm,j}) - (H_{ijn} - V_{ijn,i}-V_{ijn,j})- (H_{ijr} - V_{ijr,i}-V_{ijr,j}) \\
&\mbox{ for }\;\; i,j \neq m,n,r \\
\phi(V_{mnr,m}) &= G_{ij} - (H_{ijn}-V_{ijn,i}-V_{ijn,j}) - (H_{ijr} - V_{ijr,i}-V_{ijr,j}) \\
&= P_{ij}+F_{ij} - (H_{ijn}-V_{ijn,i}-V_{ijn,j}) - (H_{ijr} - V_{ijr,i}-V_{ijr,j}) \\
&\mbox{ for }\;\; i,j \neq m,n,r \\
\phi(\ell) &= 4\ell - \sum_i \ell_i \\
\phi(\ell_m) &= 3\ell - \sum_{i\neq m} \ell_i \\
\phi(\ell_{mn}) &= 2\ell - \ell_i - \ell_j - \ell_k - f_{ijk} \mbox{ for } i,j,k \neq m,n \\
\phi(f_{mnr}) &= h_{ij} = \ell - \ell_i - \ell_j + \ell_{ij} \mbox{ for } i,j \neq m,n,r\\
\end{align*}
\end{proposition}

\begin{proposition}
\label{X43Cremonaformulas}
\mbox{\;}
\begin{itemize}
\item[(a)] Let $D = dH - \sum_i m_i E_i - \sum_{ij} m_{ij} E_{ij} - \sum_{ijk} m_{ijk} E_{ijk}$
be a general class in $A^1(\bbX^4_3)$.
Then the Cremona image $\phi(D)$ of $D$ under the involution is
\[
\phi(D) = d'H - \sum_i m_i' E_i - \sum_{ij} m_{ij}' E_{ij} - \sum_{ijk} m_{ijk}' E_{ijk}
\]
where
\begin{align*}
d' &= \phi(D) \cdot \ell = D \cdot \phi(\ell) = D \cdot (4\ell-\sum_r \ell_r) = 4d -\sum_r m_r; \\
m_i' &= \phi(D) \cdot \ell_i = D \cdot \phi(\ell_i)
= D \cdot (3\ell - \sum_{r\neq i} \ell_r) = 3d -\sum_{r\neq i}m_r \\
m_{ij}' &= \phi(D) \cdot \ell_{ij} = D \cdot \phi(\ell_{ij}) = D \cdot (2\ell - \sum_{r \neq ij}\ell_r - f_{rst \neq ij})
= 2d - \sum_{r \neq ij} m_r - m_{rst\neq ij} \\
m_{ijk}' &= \phi(D) \cdot f_{ijk} = D \cdot \phi(f_{ijk}) = D \cdot (\ell - \sum_{r \neq ijk} \ell_r + \ell_{rs\neq ijk})
= d - \sum_{r \neq ijk} m_r + m_{rs\neq ijk}
\end{align*}

\item[(b)]
Let $T = dS - \sum_i m_i S_i
- \sum_{ij} m_{ij}P_{ij} -\sum_{ij} n_{ij}F_{ij}
- \sum_{ijk} m_{ijk} H_{ijk}
+ \sum_{ijk} (n_{ijk,i}V_{ijk,i} + n_{ijk,j}V_{ijk,j} + n_{ijk,k}V_{ijk,k})$
be a general class in $A^2(\bbX^4_3)$.
Then the Cremona image $\phi(T)$ of $T$ under the involution is
\begin{align*}
\phi(T) &= d'S - \sum_i m_i' S_i
- \sum_{ij} m_{ij}'P_{ij} -\sum_{ij} n_{ij}'F_{ij} \\
&- \sum_{ijk} m_{ijk}' H_{ijk}
+ \sum_{ijk} (n_{ijk,i}'V_{ijk,i} + n_{ijk,j}'V_{ijk,j} + n_{ijk,k}'V_{ijk,k}) \\
\end{align*}
where
\begin{align*}
d' &= \phi(T)\cdot S = T \cdot \phi(S) = T\cdot(6S - 3\sum_i S_i - \sum_{ij} P_{ij}) \\
&= 6d - 3\sum_i m_i + \sum_{ij} (m_{ij}-n_{ij}) \\
m_i' &= \phi(T) \cdot S_i = T \cdot \phi(S_i)
= T\cdot(3S - 2\sum_{r\neq i} S_r - \sum_{rs\neq i} P_{rs}) \\
&= 3d - 2\sum_{r\neq i} m_r + \sum_{rs \neq i} (m_{rs}-n_{rs}) \\
m_{ij}' &= \phi(T) \cdot F_{ij} = T \cdot \phi(F_{ij}) \\
&=T\cdot(S -S_r-S_s-S_t - P_{rs} - P_{rt} - P_{st} + 2 H_{rst} - V_{rst,r} - V_{rst,s} - V_{rst,t}) \\
&= d - m_r-m_s-m_t +m_{rs}+m_{rt}+m_{st} -n_{rs} -n_{rt}-n_{st}
+ 2m_{rst} - n_{rst,r}-n_{rst,s}-n_{rst,t} \\
n_{ij}' &= \phi(T) \cdot G_{ij} = T \cdot \phi(G_{ij}) = T\cdot(\Lambda_{rst\neq ij}) 
= T\cdot(2H_{rst}-V_{rst,r} - V_{rst,s}-V_{rst,t}) \\
&= 2m_{rst}  -n_{rst,r}-n_{rst,s}-n_{rst,t} \\
m_{ijk}' &= \phi(T) \cdot H_{ijk} = T \cdot \phi(H_{ijk}) \\
&= T\cdot(2G_{rs} - (H_{rsi} - V_{rsi,r}-V_{rsi,s}) - (H_{rsj} - V_{rsj,r}-V_{rsj,s})- (H_{rsk} - V_{rsk,r}-V_{rsk,s}) \\
&\mbox{ for }\;\; rs \neq ijk  \\
&= 2n_{rs} - (m_{rsi} - n_{rsi,r}-n_{rsi,s}) - (m_{rsj} - n_{rsj,r}-n_{rsj,s})- (m_{rsk} - n_{rsk,r}-n_{rsk,s} \\
n_{ijk,i}' &= \phi(T) \cdot V_{ijk,i} = T \cdot \phi(V_{ijk,i}) \\
&= T\cdot(G_{rs} - (H_{rsj}-V_{rsj,r}-V_{rsj,s}) - (H_{rsk} - V_{rsk,r}-V_{rsk,s}) ) \\
&= n_{rs} - (m_{rsj}-n_{rsj,r}-n_{rsj,s}) - (m_{rsk} - n_{rsk,r}-n_{rsk,s}) \\
\end{align*}

\item[(c)] Let $C = d\ell - \sum_i m_i\ell_i - \sum_{ij} m_{ij}\ell_{ij} - \sum_{ijk}m_{ijk} f_{ijk}$ be a general class in $A^3(\bbX^4_3)$.
Then the Cremona image $\phi(C)$ of $C$ under the involution is
\[
\phi(C) = d' \ell - \sum_i m_i' \ell_i - \sum_{ij} m_{ij}' \ell_{ij} - \sum_{ijk} m_{ijk}' f_{ijk}
\]
where
\begin{align*}
d' &= \phi(C) \cdot H = C \cdot \phi(H) = C \cdot (4H -3\sum_i E_i - 2\sum_{ij} E_{ij} - \sum_{ijk} E_{ijk}) \\
&= 4d - 3\sum_i m_i - 2\sum_{ij} m_{ij} - \sum_{ijk} m_{ijk}; \\
m_i' &= \phi(C) \cdot E_i = C \cdot \phi(E_i)
= C \cdot (H - \sum_{r\neq i} E_r - \sum_{rs \neq i} E_{rs} - \sum_{rst\neq i} E_{rst}) \\
&= d - \sum_{r\neq i} m_r - \sum_{rs \neq i} m_{rs} - \sum_{rst\neq i} m_{rst}; \\
m_{ij}' &= \phi(C) \cdot E_{ij} = C \cdot\phi(E_{ij}) = C \cdot E_{rst\neq ij} = m_{rst\neq ij} \\
m_{ijk}' &= \phi(C) \cdot E_{ijk} = C \cdot\phi(E_{ijk}) = C \cdot E_{rs\neq ijk} = m_{rs\neq ij} \\
\end{align*}

\end{itemize}

\end{proposition}

We note that, for surface classes in $A^2(\bbX^4_3)$,
higher-dimensional contact is observed by having nonzero coefficients
in the $F$, $H$, and $V$ basis elements.
For curve classes in $A^3$, this higher-dimensional contact
corresponds to nonzero coefficients in the $\ell_{ij}$ and the $f_{ijk}$ basis elements
(corresponding to a curve meeting a coordinate line or a coordinate plane).
The formulas above show that a similar phenomenon happens as in the $\bbP^3$ case:
if these are all zero, that is preserved under the involution.

\begin{corollary}\label{P4multiplicityonly}
\begin{itemize}
\item[(a)]
The subspace of $A^2(\bbX^4_3)$ spanned by $S$, the $S_i$, and the $P_{ij}$
is invariant under the Cremona involution $\phi$.
If $T = dS - \sum_i m_i S_i - \sum_{ij} m_{ij}P_{ij}$ is an element in this subspace, then
$\phi(T) = d'S - \sum_i m_i' S_i - \sum_{ij} m_{ij}'P_{ij}$
where
\begin{align*}
d' &= 6d - 3\sum_i m_i + \sum_{ij} m_{ij} \\
m_i' &= 3d - 2\sum_{r\neq i} m_r + \sum_{rs \neq i} m_{rs} \\
m_{ij}' &= d - m_r-m_s-m_t +m_{rs}+m_{rt}+m_{st} \;\;\;\mbox{ for }\;\;\; r,s,t \neq i,j \\
\end{align*}
\item[(b)]
The subspace of $A^3(\bbX^4_3)$ spanned by $\ell$, $\ell_i$
is invariant under the Cremona involution $\phi$.
If $C = d\ell - \sum_i m_i \ell_i $ is an element in this subspace, then
$\phi(C) = d'\ell - \sum_i m_i' \ell_i$
where
\begin{align*}
d' &= 4d - 3\sum_i m_i  \\
m_i' &= d - \sum_{r\neq i} m_r  \\
\end{align*}
\end{itemize}
\end{corollary}

For divisors, the natural subspace invariant under the involution
is that generated by the $E_{ij}$'s and $E_{ijk}$'s.
If we are only interested in the multiplicity conditions at the points,
we can therefore mod out by this subspace of $A^1$, and obtain the following.

\begin{corollary}\label{P4multiplicitypointsonlyA1}
The subspace of $A^1(\bbX^4_3)$ spanned by the $E_{ij}$'s and $E_{ijk}$'s
is invariant under the Cremona involution $\phi$.
Denote by $\bar{A}^1$ the quotient of $A^1$ by this subspace;
the involution $\phi$ descends to an involution of $\bar{A}^1$.
If $\bar{D} = dH - \sum_i m_i E_i$ represents a coset in this subspace, then
$\phi(\bar{D}) = d'H - \sum_i m_i' E_i$
where
\[
d' = 4d - \sum_i m_i \;\;\;\text{ and }\;\;\;
m_i' = 3d - \sum_{r\neq i} m_r.
\]
\end{corollary}

\section{Six and seven points in $\bbP^4$}
The formulas for how degrees and multiplicities change for curves, surfaces, and divisors in $\bbP^4$
under the standard Cremona transformation can be used to analyze compositions of such Cremona transformations
based at more than five points.
We will present the orbits of the linear subspaces spanned by subsets of the points in this section.

If we first consider six general points in $\bbP^4$,
it is easy to see using the formulas above that any line through $2$ of the six points,
$2$-plane through $3$ of them, or a hyperplane through $4$,
is either contracted by the Cremona transformation
or is sent to itself.

The case of seven general points in $\bbP^4$ is one step more interesting.
In this case, for a line through two of the seven points,
it is either contracted by the Cremona transformation based at five of the points
(if the two points are a subset of the five),
is sent to itself (if one of the two is a subset of the five)
or is sent to the rational normal quartic (RNQ) through all seven points
(if neither of the two is among the five).
 
 The iteration of Cremona now leads us to consider the transformation of the RNQ;
 applying Cremona at any five yields back the line joining the other two
 (since the Cremona is an involution).
 
 Hence the Cremona orbit of the line through two points is the collection of all of the $21$ lines,
 plus the rational normal quartic through all seven points.

Now consider the $2$-plane spanned by three of the $7$ points.
Performing a Cremona transformation at $5$ of the $7$ points, we see that
if all three points are among the $5$, the plane is contracted as part of the fundamental locus.
If two of the three points are among the five, the plane is sent to itself.
If only one of the three points are among the five, then the Cremona image
is a surface of degree three,
with a point of multiplicity $3$ at that one point,
and multiplicity $1$ at the other six points.
It contains the line joining that one point to the other six, with multiplicity one each, and no other lines joining the points.
It also contains the RNQ with multiplicity one.
This cubic surface is a cone over a twisted cubic in $\bbP^3$.

Iterating the Cremona by applying it to this cone, we see that
if the five points contain the vertex, it will be transformed back into the $2$-plane.
If it does not, it is preserved.

Hence the Cremona orbit of the $2$-plane through $3$ points in $\bbP^4$ consists of the
$35$ planes and the $7$ cubic cones.

For the hyperplanes through $4$ of the seven points,
there are four cases to consider.
We choose five of the seven to perform the Cremona transformation at.
If all $4$ of the hyperplane points are among the five, then the hyperplane is contracted to a point.
If $3$ of the hyperplane points are among the five, then the hyperplane is transformed to another hyperplane.
If $2$ of the hyperplane points are among the five, it is transformed
into a quadric double cone: a cone over a smooth conic with vertex a line
(the line corresponding to the two points).
To be explicit, take the line joining the two points, and a complementary plane;
projection from the line to the plane sends the other five points to five general points in the plane,
and there is a unique conic in that plane through those five points.
The threefold is obtained as the cone over the conic with vertex the line.
The surfaces contains all the lines joining the two points with the other five,
as well as containing the RNQ too.

If we apply a second Cremona transformation to this quadric, we either
return to the hyperplane, preserve the quadric, or (if we use as the base points the five points
not on the vertex line) we obtain a cubic surface double at all seven points.
It is also double all along the RNQ; this cubic surface is the secant variety to the RNQ, in fact.

Further applications of Cremona to this cubic surface lowers the degree and returns us to the quadric double cone; we see then that the orbit of the hyperplane consists of the set of $35$ hyperplanes, the $21$ quadric double cones, and the cubic secant variety to the RNQ.

It is interesting that the two special linear systems with irreducible members in $\bbP^4$
imposing only double points appear here:
the quadrics double at two points and the cubics double at $7$.

\section{Eight points in $\bbP^4$}
We now consider the case of Cremona transformations based at $8$ general points $p_1,\dots,p_8$ in $\bbP^4$.
Denote by $L_{ij}$ the line joining $p_i$ and $p_j$ as usual.
Denote by $Q_i$ the rational normal quartic curve passing through all eight points except $p_i$
(i.e., passing through the other $7$).

It is easy to see, with a parallel computation as that done above for seven points,
that the orbit of a line through two points, say $L_{12}$,
consists of all $28$ such lines $L_{ij}$,
and all $8$ of the RNQ's $Q_k$.

We can now take up the case of surfaces, which is more involved.
We will record the data for a surface of degree $d$, having multiplicity $m_i$ at $p_i$,
multiplicity $n_i$ along $Q_i$, and multiplicity $m_{ij}$ along $L_{ij}$,
by the triangular array of numbers:
\begin{equation}\label{P4surface8points}
\begin{matrix}
d & m_1 & m_2 & m_3 & m_4 & m_5 & m_6 & m_7 & m_8 \\
  & n_1 & n_2 & n_3 & n_4 & n_5 & n_6 & n_7 & n_8 \\
  &    & m_{12} & m_{13} & m_{14} & m_{15} & m_{16} & m_{17} & m_{18} \\
    &    &  & m_{23} & m_{24} & m_{25} & m_{26} & m_{27} & m_{28} \\
    &    &  &   & m_{34} & m_{35} & m_{36} & m_{37} & m_{38} \\
    &    &  &   &   & m_{45} & m_{46} & m_{47} & m_{48} \\
    &    &  &   &   &  & m_{56} & m_{57} & m_{58} \\
    &    &  &   &   &  &  & m_{67} & m_{68} \\
    &    &  &   &   &  &  &  & m_{78} \\
\end{matrix}
\end{equation}

Suppose we perform the five-point Cremona on the first five points $1,2,3,4,5$.
Then the degree $d$, the multiplicities $m_i$ for $i \leq 5$, and the $m_{ij}$ for $i,j \leq 5$,
are transformed as indicated in Corollary \ref{P4multiplicityonly}(a).

For multiplicity $m_{ij}'$ with $i \leq 5$ and $j \geq 6$,
we note that this line $L_{ij}$ is left invariant under the Cremona,
so that $m_{ij}' = m_{ij}$ for these indices.

For multiplicities $m_{ij}$ with both $i,j \geq 6$,
we note that this $L_{ij}$ is the image of $Q_k$ where $\{i, j,k\} = \{6,7,8\}$;
$k$ is the third index.
Hence $m_{ij}' = n_k$ for $k = \{6,7,8\}-\{i,j\}$.

For the $n_k'$ with $k \geq 6$, conversely we have $n_k' = m_{ij}$ where $i,j = \{6,7,8\}-\{k\}$.
For $n_k'$ with $k\leq 5$, since such a $Q_k$ is fixed, we have $n_k' = n_k$.
This gives the following:

\begin{corollary}\label{surfacecremonaformulas}
The surface with degree and multiplicities indicated by (\ref{P4surface8points}) is transformed,
under the Cremona involution based at the first five points $p_1,p_2,p_3,p_4,p_5$,
into the surface with degree and multiplicities recorded by:
\[
\begin{matrix}
d' & m_1' & m_2' & m_3' & m_4' & m_5' & m_6 & m_7 & m_8 \\
  & n_1 & n_2 & n_3 & n_4 & n_5 & m_{78} & m_{68} & m_{67} \\
  &    & m_{12}' & m_{13}' & m_{14}' & m_{15}' & m_{16} & m_{17} & m_{18} \\
    &    &  & m_{23}' & m_{24}' & m_{25}' & m_{26} & m_{27} & m_{28} \\
    &    &  &   & m_{34}' & m_{35}' & m_{36} & m_{37} & m_{38} \\
    &    &  &   &   & m_{45}' & m_{46} & m_{47} & m_{48} \\
    &    &  &   &   &  & m_{56} & m_{57} & m_{58} \\
    &    &  &   &   &  &  & n_8 & n_7 \\
    &    &  &   &   &  &  &  & n_6 \\
\end{matrix}
\]
where
\begin{align*}
d' &= 6d - 3\sum_{i=1}^5 m_i + \sum_{1\leq i < j \leq 5} m_{ij} \\
m_i' &= 3d - 2\sum_{r\leq 5; r\neq i} m_r + \sum_{r,s \leq 5; r,s \neq i} m_{rs} \;\;\text{for}\;\; i \leq 5\\
m_{ij}' &= d - m_r-m_s-m_t +m_{rs}+m_{rt}+m_{st}  \;\;\text{for}\;\; i,j \leq 5\;\;\mbox{and}\;\; r,s,t = \{1,2,3,4,5\}-\{ i,j\} \\
\end{align*}
\end{corollary}

The Proposition below presents the orbit of $L_{123}$, a $2$-plane through three of the points, in (b).
For notational consistency with the other surfaces in this orbit, we will also denote $L_{ijk}$ by $S_1(ijk)$.
We have included in (a) the remarks above about the orbit of the line $L_{12}$.
In (c) we present the orbit of a hyperplane; the reader can verify the computations as an exercise.

\begin{proposition}\label{P4orbits}
Fix $8$ general points in $\bbP^4$,
and consider Cremona transformations based at $5$ of the $8$, in series.
\begin{itemize}
\item[(a)]
The orbit of a line through two of the $8$ points
consists of the $28$ lines $L_{ij}$ ($1 \leq i < j \leq 8$) through two ($p_i$ and $p_j$) of the $8$ points,
and the $8$ rational normal quartics $Q_k$ ($1 \leq k \leq 8$
through $7$ of the $8$ points (through all seven except $p_k$).
\item[(b)]
The orbit of a plane through three of the $8$ points
consists of:
\begin{itemize}
\item[(b1)]
the $56$ planes $L_{ijk} = S_1(1jk)$ through three of the $8$ points (namely $p_i$, $p_j$, and $p_k$);
the plane $L_{123}=S_1(123)$ is recorded as
\[
\begin{matrix}
1 & 1 & 1 & 1 & 0 & 0 & 0 & 0 & 0 \\
   & 0 & 0 & 0 & 0 & 0 & 0 & 0 & 0 \\
   &    & 1 & 1 & 0 & 0 & 0 & 0 & 0 \\
   &    &    & 1 & 0 & 0 & 0 & 0 & 0 \\
   &    &    &    & 0 & 0 & 0 & 0 & 0 \\
   &    &    &    &    & 0 & 0 & 0 & 0 \\
   &    &    &    &    &    & 0 & 0 & 0\\
   &    &    &    &    &    &    & 0 & 0 \\
   &    &    &    &    &    &    &    & 0 \\
\end{matrix}
\]
\item[(b2)]
the $56$ surfaces $S_3(i,j)$ of degree $3$ with one point $p_i$ of multiplicity $3$, $6$ points of multiplicity one, and one point $p_j$ of multiplicity $0$.
It contains the lines joining the triple point $p_i$ to all other multiplicity one points $p_k$ ($k \neq j$) and no other lines;
it contains the rational normal quartic $Q_j$ through the triple point and the six multiplicity one points.
For example, $S_3(81)$ is recorded as:
\[
\begin{matrix}
3 & 0 & 1 & 1 & 1 & 1 & 1 & 1 & 3 \\
   & 1 & 0 & 0 & 0 & 0 & 0 & 0 & 0 \\
   &    & 0 & 0 & 0 & 0 & 0 & 0 & 0 \\
   &    &    & 0 & 0 & 0 & 0 & 0 & 1 \\
   &    &    &    & 0 & 0 & 0 & 0 & 1 \\
   &    &    &    &    & 0 & 0 & 0 & 1 \\
   &    &    &    &    &    & 0 & 0 & 1\\
   &    &    &    &    &    &    & 0 & 1 \\
   &    &    &    &    &    &    &    & 1 \\
\end{matrix}
\]
\item[(b3)]
the $56$ sextic surfaces $S_6(ijk)$ of degree $6$ with three points ($p_i$, $p_j$, $p_k$) of multiplicity one,
and the other $5$ points of multiplicity $3$.
It contains the lines joining any two of the multiplicity $3$ points and no other lines;
It contains the rational normal quartics through the five multiplicity $3$ points and any two of the three multiplicity one points.
For example, $S_6(678)$ is recorded as:
\[
\begin{matrix}
6 & 3 & 3 & 3 & 3 & 3 & 1 & 1 & 1 \\
   & 0 & 0 & 0 & 0 & 0 & 1 & 1 & 1 \\
   &    & 1 & 1 & 1 & 1 & 0 & 0 & 0 \\
   &    &    & 1 & 1 & 1 & 0 & 0 & 0 \\
   &    &    &    & 1 & 1 & 0 & 0 & 0 \\
   &    &    &    &    & 1 & 0 & 0 & 0 \\
   &    &    &    &    &    & 0 & 0 & 0\\
   &    &    &    &    &    &    & 0 & 0 \\
   &    &    &    &    &    &    &    & 0 \\
\end{matrix}
\]
\item[(b4)]
the $28$ surfaces $S_{10}(ij)$ of degree $10$ with two points ($p_i$ and $p_j$) of multiplicity $6$
and the other $6$ points of multiplicity $3$.
It contains the lines joining the multiplicity one points to the multiplicity six points (each with multiplicity one)
and the line joining the two multiplicity $6$ points with multiplicity $3$.
It contains the $6$ rational normal quartics that pass through the two multiplicity $6$ points and five of the six multiplicity one points.
For example, $S_{10}(78)$ is recorded as:
\[
\begin{matrix}
10 & 3 & 3 & 3 & 3 & 3 & 3 & 6 & 6 \\
    & 1 & 1 & 1 & 1 & 1 & 1 & 0 & 0 \\
    &    & 0 & 0 & 0 & 0 & 0 & 1 & 1 \\
    &    &    & 0 & 0 & 0 & 0 & 1 & 1 \\
    &    &    &    & 0 & 0 & 0 & 1 & 1 \\
    &    &    &    &    & 0 & 0 & 1 & 1 \\
    &    &    &    &    &    & 0 & 1 & 1 \\
    &    &    &    &    &    &    & 1 & 1 \\
    &    &    &    &    &    &    &    & 3 \\
\end{matrix}
\]
\item[(b5)]
the $8$ surfaces $S_{15}(i)$ of degree $15$ with one point ($p_i$) of multiplicity $3$ and the other seven points of multiplicity $6$.
It contains the joining any two points of multiplicity $6$, and no other lines.
It contains all $8$ of the rational normal quartics; the one through the seven multiplicity $6$ points with multiplicity three, and all others with multiplicity one.
For example, $S_{15}(1)$ is recorded as:
\[
\begin{matrix}
15 & 3 & 6 & 6 & 6 & 6 & 6 & 6 & 6 \\
    & 3 & 1 & 1 & 1 & 1 & 1 & 1 & 1 \\
    &    & 0 & 0 & 0 & 0 & 0 & 0 & 0 \\
    &    &    & 1 & 1 & 1 & 1 & 1 & 1 \\
    &    &    &    & 1 & 1 & 1 & 1 & 1 \\
    &    &    &    &    & 1 & 1 & 1 & 1 \\
    &    &    &    &    &    & 1 & 1 & 1 \\
    &    &    &    &    &    &    & 1 & 1 \\
    &    &    &    &    &    &    &    & 1 \\
\end{matrix}
\]
\end{itemize}
\item[(c)] We use the notation that $(d;m_1m_2\cdots m_8)$ represents a hyperplane of degree $d$
having multiplicity $m_i$ at $p_i$.
The orbit of the hyperplane through the first four points (represented by $(1;11110000)$
consists of the following divisors,
and all related divisors obtained by permutations of the eight points:
\begin{align*}
&(1;11110000) &
&(2;22111110) &
&(3;22222220) &
&(3;32222111) \\
&(4;33332221) &
&(4;43222222) &
&(5;44333322) &
&(6;44444432) \\
&(6;54443333) &
&(7;55544443) &
&(7;64444444) &
&(8;65555544) \\
&(9;66665555) & 
&(10;76666666)& & \\
\end{align*}
\end{itemize}
\end{proposition}

\section{Applications}\label{applicationssection}

\begin{proposition}\label{intersection}
Let $R$ and $T$ be two Weyl planes on $X_{4,8}$. Then $R \cdot T \in \{0,1,3\}$.
\end{proposition}

\begin{proof}
If we choose an element $w$ of the Weyl group that sends Weyl plane $R$ to the actual plane $S_1(123)$,
then since the intersection form is preserved we have $R \cdot T = S_1(123) \cdot w(T)$.
Hence it suffices to show that the intersection of $S_1(123)$ with any Weyl plane is in $\{0,1,3\}$.
This one can check by hand for all of the cases.

Even easier would be to notice that, if $\phi$ is the Cremona transformation centered at the first five points,
then by Corollary \ref{surfacecremonaformulas} we have
$\phi(S_1(123)) = -P_{45}$ in the Chow ring.
Hence it also suffices to show that $-P_{45} \cdot T \in \{0, 1, 3\}$ for all Weyl planes $T$.
By Proposition \ref{X43}, intersecting with $-P_{45}$ picks out exactly the multiplicity $m_{45}$ for the Weyl plane.
Hence it suffices, after taking account of permutations,
to observe that for all Weyl planes, all $m_{ij}$ are in $\{0,1,3\}$.
\end{proof}

\begin{proposition}\label{twoweylplanes}
Let $R$ and $T$ be any Weyl planes on $X^{4}_{8}$.
If $R \cdot T \neq 3$, then there exists $w$ in the Weyl group of $X^{4}_{8}$
such that $w(R)=H_{123}$ and $w(T)=H_{456}$.
\end{proposition}

\begin{proof}
It is enough to prove the statement for $R \neq T$.
One can use the same technique as in Proposition \ref{intersection}
and reduce one Weyl surface to $-P_{45}$ and select Weyl surfaces from the list of Proposition \ref{P4orbits}(b)
that have $m_{45}\in \{0,1\}$.
Then applying the Cremona transformation $\phi$ centered at the first five points,
we have the first Weyl surface being $S_1(123)$
and the other on the following lists (up to permutations that fix $\{1,2,3\}$):

\begin{enumerate}
\item Case $S_1(123)\cdot T = 1$:

\begin{enumerate}
\item $S_1(123) \cdot S_1(456) =1$
\item $S_1(123) \cdot S_3(4, 1)=1$ 
\item $S_1(123) \cdot S_6(126)=1$
\item $S_1(123) \cdot S_{10}(45)=1$
\item $S_1(123) \cdot S_{15}(1)=1$
\end{enumerate}

We are done in the first case of course.
In the other cases it suffices to find five indices, two of them among $\{1,2,3\}$,
so that the corresponding Cremona transformation reduces the degree of the second surface;
such a Cremona will fix $S_1(123)$ and we proceed then by induction on the degree.

To reduce the cubic surface, $\{2,3,4,7,8\}$ will work;
for the sextic, $\{1,3,5,7,8\}$ works.
For the surface of degree $10$, $\{1,2,4,5,6\}$ suffices;
finally for the last surface of degree $15$, $\{2,3,6,7,8\}$ works.


\item Case $S_1(123) \cdot T = 0$:
In this case a similar approach yields the following lists to analyze:
\begin{enumerate}
\item $S_1(123) \cdot S_1(145) =0$ or $S_1(123) \cdot S_1(124)=0$ 
\item $S_1(123) \cdot S_{3}(1, 2)=0$ or $S_1(123) \cdot S_{3}(1, 4)=0$ or $S_1(123) \cdot S_{3}(4,5)=0$ 
\item $S_1(123) \cdot S_{6}(145)=0$ or $S_1(123) \cdot S_{6}(456)=0$
\item $S_1(123) \cdot S_{10}(12)=0$ or $S_1(123) \cdot S_{10}(15)=0$ 
\item $S_1(123) \cdot S_{15}(4)=0$
\end{enumerate}

The same proof as in the prior case works; in each situation one finds five indices, two among $\{1,2,3\}$,
that reduce the degree of the second surface.
For example, $\{2,3,6,7,8\}$ works for the degree $15$ surface.
We leave the details of the other cases to the reader.


\end{enumerate}
\end{proof}

We remark that Weyl planes that intersect in three points (modulo permutations of points) are
$$
S_1(123) \cdot S_6(123)=S_3(1,8) \cdot S_3(8,1) =3.
$$


\begin{corollary}\label{disj}
Assume $R$ and $T$ are Weyl planes in the base locus of the linear system $|D|$
for an effective divisor $D= dH-\sum_{i=1}^8 m_i E_i$ on $X^{4}_{8}$.
Then $R \cdot T=0$.
\end{corollary}

\begin{proof}
We argue by contradiction.
Assume first that $R \cdot T=1$.
By Proposition \ref{twoweylplanes}, we can apply a series of Cremona transformations,
which do not change the hypothesis on the base locus,
and assume that $R=S_1(123)$ and $T=S_1(456)$.
It follows from the results of \cite{cddgp}, Section 4, and \cite{dp1}, Proposition 4.2, that we therefore have
\[
m_1+m_2+m_3-2d>0
\;\;\text{and}\;\;
m_4+m_5+m_6-2d>0.
\]
Hence the system of rational normal curves of degree 4 passing through first 6 points must be in the base locus of $|D|$;
since this family of curves covers $\bbP^4$, we conclude $|D|$ is empty, a contradiction.


If the two Weyl planes intersect in three points,
then they are either $S_3(1,8)$ and  $S_3(8,1)$ or $S_1(123)$ and $S_6(123)$ (up to permutations).
We will analyze the first case; the other is handled by a similar argument.
Assume by contradiction that both such Weyl planes are in the base locus of the linear system $|D|$ of an effective divisor $D$.
By Proposition 3 of \cite{bdp4},
the multiplicity of containment of the surface $S_3(1,8)$ in the base locus of a divisor $D$ is $2m_1+m_2+\ldots +m_7-5d<0$;
therefore since both $S_3(1,8)$ and $S_3(8,1)$ are in the base locus we obtain
$2(m_1+\ldots +m_8)-10d<0$.
This contradicts the effectivity because $2(m_1+\ldots+m_6)+m_7+3m_8-10d\leq 2(m_1+\ldots +m_8)-10d<0$;
therefore a family of curves of degree $10$ with six double points, one simple point and one triple point meets $D$ negatively,
and so is part of the base locus also. Corollary \ref{P4multiplicityonly} implies that these curves are in the Weyl orbit
of a line through a point,
and therefore again cover the projective space, a contradiction.
The remaining case can be handled by the same argument.
\end{proof}

\begin{remark}
In fact, the linear equations of pencils of curves in the base locus of the linear system of an effective divisor $D$,
that in this case are equivalent to two Weyl planes that meet in the base locus of $|D|$,
give the {\it faces of the cone of effective divisors}.
We will prove this theorem in the case of a Mori Dream Space in arbitrary dimension in \cite{DM}.
\end{remark}

\begin{remark}
In \cite{DM} we prove that a Weyl curve and a Weyl divisor that meet
can not be simultaneously in the base locus of the linear system of an effective divisor $D$.
\end{remark}

For any effective divisor $D\in Pic(X^{4}_{8})$,
define $\widetilde{D}\in Pic(\widehat{X^{4}_{8}})$
to be the proper transform of $D$ after blowing up the points, Weyl lines, and Weyl planes in the base locus of $|D|$
to obtain $\widehat{X^{4}_{8}}$.
Corollary \eqref{disj} proves that the space $\widehat{X^{4}_{8}}$ is smooth.

We remark first that the Weyl line $C$ has normal bundle $\oplus \mathcal{O}(-1)^{3}$.
If $D \cdot C <0$ then the Weyl line $C$ is in the base locus of the linear system $|D|$.
Let  $D_{(1)}$ denote the proper transform of $D$ under the blow up $Y$ of all fixed Weyl lines in $X^{4}_s$.
For each Weyl line $C$, define $k_C = -D \cdot C$.
It was proved in \cite{bdp1} that if $k_C > 0$ then
\begin{proposition}\label{prop}
If $D$ be an effective divisor on $X^4_{8}$, then
\[
h^1(X^{4}_s, \mathcal{O}_{X^{4}_s}(D)) =
\sum_{C} {{2+k_{C}}\choose 4} + h^1(Y,  \mathcal{O}_{Y}(D_{(1)}) -  h^2(Y,  \mathcal{O}_{Y}(D_{(1)}).
\]
\end{proposition}
A general form of Proposition \ref{prop} for $(-1)$ curves in arbitrary dimenison will be given in \cite{DM}.
We conclude that if $k_C\geq 2$ then $h^1(X^{4}_{s}, \mathcal{O}_{X^{4}_{s}}(D)) \geq 1+ h^1(Y,  \mathcal{O}_{Y}(D_{(1)}) -  h^2(Y,  \mathcal{O}_{Y}(D_{(1)}) $.

\begin{conjecture}\label{conj2} Let $D$ be an effective divisor on $X^{4}_s$,
with $H^1(X^{4}_s, \mathcal{O}_{X^{4}_s}(D))=0$.
Then $D \cdot C\geq -1$ for any Weyl line $C$.
\end{conjecture}

\begin{remark}
For arbitrary number of points $s$, the converse of Conjecture \eqref{conj2} is not true.
Indeed, take $D:=4H -2\sum_{i=1}^{14} E_i \in Pic(X^4_{14})$.
We can see that $D \cdot C\geq 0$ for any Weyl line $C$; however the Alexander Hirschowitz Theorem implies that

$$h^1(X^{4}_s, \mathcal{O}_{X^{4}_s}(D))=1.$$

\end{remark}

For every $r$-subset $I(r)$ of the indices $\{1,\ldots, 8\}$, let $L_{I(r)}$ be the linear span of the corresponding points.
Let $k_{w(L_{I(r)})}$ be the multiplicity of containment of the Weyl cycle $w(L_{I(r)})$
in the base locus of $D$, for a Weyl group element $w$.
In \cite{bdp4} the{\it  Weyl expected dimension} for an effective divisor $D$ was introduced as

$$
wdim(D):=
\chi(D)+ \sum_{r= 1}^{3}\sum_{I(r) \in \{1,\ldots, 8\}} \sum_{w\in W} (-1)^{r+1} {{4+k_{w(L_{I(r)})}-r-1}\choose 4}.
$$

Morever, in \cite{bdp4} it was conjectured that for $D$ every effective divisor on $\widehat{X^{4}_{8}}$,  then dimension of space of global section of $D$ equals the {\it Weyl expected dimension}.

\begin{conjecture} \label{conj} Let $D$ be an effective divisor on  $\widehat{X^{4}_{8}}$.
\begin{enumerate}
\item If $D \cdot C\geq -1$ for all Weyl curves $C$ then $$H^1(X^{4}_s, \mathcal{O}_{X^{4}_s}(D))=0.$$
\item $
h^{0}(D)=wdim (D) + \sum_{r=1}^{3}(-1)^{r+1} h^r(\widetilde{D}).
$
\item For every $r\geq 1$,  $h^r(\widetilde{D})=0$. 
\item Moreover, $\widetilde{D}$ is globally generated on $\widehat{X^{4}_{8}}$.
\end{enumerate}
\end{conjecture}

We remark that Conjecture \eqref{conj} part (2) implies $wdim(D)=\chi(\widetilde{D})$, while part (3) implies that conjecture of \cite{bdp4} regarding dimension $h^0(D)$ is true.

\begin{remark}
We remark that Conjecture \ref{conj} holds for effective divisors on $X^{n}_{n+2}$
\cite{bdp1}, \cite{dp1}, \cite{dp2};
therefore it holds for $X^{4}_{6}$.
Notice that $\widehat{X^{4}_{9}}$ is not a Mori Dream Space and in fact, there are infinitely many Weyl lines.
The authors believe that the Conjecture \ref{conj} also holds for $\widehat{X^{4}_{9}}$
with a similar construction for the Weyl planes as the one presented here. 
\end{remark}

\begin{remark}
Conjecture \ref{conj} fails in $\widehat{X^{4}_{10}}$,
because for arbitrary number of points, in non Mori-dream spaces Weyl cycles are not the only obstructions.
Indeed, consider the divisor
$$D:=4H-4E_1-2\sum_{i=2}^{10} E_i$$.
We remark that $D$ contains in the base locus of its linear system just double lines $k_{L_{1i}}=2$;
therefore its proper transform under the blow up of all its Weyl base locus (i.e. only lines) is
$$\widehat{D}:=4H- 4E_1- 2\sum_{i=2}^{10} E_i -2 \sum_{i=2}^{10} E_{1i}$$ 
Moreover, since $k_{L_{1i}}=2$ we have
$$\chi(D)={{4+4}\choose 4}-{{4+4-1}\choose 4} - 9 {{4+2-1}\choose 4}  =70-35-45=-10$$
$$wdim(D)=\chi(\widehat{D})=\chi(D)+9{{2+2}\choose 4}=-1$$
However, this divisor is effective,
and in fact the Alexander - Hirschowitz theorem implies that it is unique in its linear system.
We conclude that $h^{0}(D)=1\neq 0=wdim(D)$, therefore $h^1(\widehat{D})=1$.
\end{remark}

\end{document}